\documentclass[12pt]{article}
\usepackage{amsmath, amssymb, amscd, amsthm, amsfonts, xfrac, mathrsfs,diagbox,float}
\numberwithin{equation}{section}
\usepackage{graphicx}
\usepackage{hyperref}
\usepackage{authblk}
\usepackage{url}
\hypersetup{backref, colorlinks=true}
\newtheorem{theorem}{Theorem}[section]

\newtheorem{lemma}[theorem]{Lemma}

\title{\textbf{Asymptotic normality of coefficients of P-recursive polynomial sequences}}
\author{Zhongjie Li}
\affil{School of Mathematics and KL-AAGDM \break Tianjin University \break Tianjin 300350, China \break\texttt{lizhongjie@tju.edu.cn}}
\date{}

\begin{document}
\maketitle
\begin{abstract}
In recent years, the asymptotic normality of some famous combinatorial sequences has been  the subject of extensive study. However, the methods used to prove the asymptotic normality of various combinatorial sequences differ significantly. In this paper, we present a sufficient condition for establishing the asymptotic normality of the coefficients of a general P-recursive polynomial sequence. Additionally, we provide two examples that illustrate the application of this sufficient condition.
\end{abstract}


\section{Introduction}
Suppose that $\{f_n(x)\}_{n\ge 0}$ is a polynomial sequence with nonnegative coefficients $a(n,k)$, that is,
\begin{align}\label{1.1}
f_n(x) = \sum_{k=0}^{n} a(n,k) x^k.
\end{align}
Let
\begin{align*}
p(n,k) = \frac{a(n,k)}{\sum_{j=0}^{n}a(n,j)}
\end{align*}
be the normalized probabilities.
We say that the coefficient $a(n,k)$  is {\it{asymptotically normal with mean $\mu_{n}$ and variance $\sigma_{n}^{2}$ by a central limit theorem}} if
\begin{align}\label{1.2}
\lim_{n\to\infty}\sup_{x\in\mathbb{R}}\left|\sum_{k\leq\mu_n+x\sigma_n}p(n,k)-\frac{1}{\sqrt{2\pi}}\int_{-\infty}^x\exp(-t^2/2)dt\right|=0.
\end{align}
We say that the coefficient $a(n,k)$  is {\it{asymptotically normal with mean $\mu_{n}$ and variance $\sigma_{n}^{2}$ by a local limit theorem}} on the real set $\mathbb{R}$ if
\begin{align}\label{1.3}
\lim_{n\to\infty}\sup_{x\in\mathbb{R}}\left|\sigma_np(n,\lfloor\mu_n+x\sigma_n\rfloor)-\frac{1}{\sqrt{2\pi}}e^{-x^2/2}\right|=0.
\end{align}
It is obvious that (\ref{1.3}) can lead to (\ref{1.2}), see Bender \cite{bender1973central} and Canfield \cite{canfield1977central} for details.\\
\indent In recent years, the asymptotic normality of various combinatorial sequences has been extensively studied. This includes the Eulerian numbers \cite{carlitz1972asymptotic},  the coefficients of $q$-Catalan numbers \cite{chen2008limiting}, the $q$-derangement numbers \cite{chen2010limiting}, the Stirling numbers of the first kind \cite{feller1945fundamental}, the Stirling numbers of the second kind \cite{harper1967stirling}, the resonant sextet numbers \cite{li2021analytic}, the Laplacian coefficients of graphs \cite{wang2017asymptotic} and the Baxter permutations \cite{zhao2024asymptotic}. Most of their proofs are given based on the properties of specific sequences. Next, we will present a systematic method to prove the asymptotic normality of P-recursive polynomial sequences $\{f_{n}(x)\}_{n\ge 0}$.\\
\indent Recall that a {\it P-recursive sequence $\{a_n\}_{n \ge 0}$ of order $d$} satisfies a recurrence relation of the form 
\begin{align*}
p_0(n) a_n + p_1(n) a_{n+1} + \cdots + p_d(n) a_{n+d} = 0,\quad\forall n\ge 1,
\end{align*}
where $p_i(n)$ are polynomials in $n$ (see \cite[Section 6.4]{stanley1999enumerative}). Wimp and Zeilberger \cite{zeilberger1985asymptotics} (see also \cite[Sec. VIII.7]{flajolet2009analytic}) showed that a P-recursive sequence $\{a_n\}_{n \ge 0}$ is asymptotically equal a linear combination of terms in the form of
\begin{align}\label{1.4}
e^{Q(\rho, n)} s(\rho, n),
\end{align}
where
\begin{align}
Q(\rho, n) & = \mu_0 n \log n + \sum_{j = 1}^\rho\mu_j n^{j/\rho}, \label{1.5}\\
s(\rho, n) & = n^{r} \sum_{j = 0}^{t - 1} (\log n)^{j} \sum_{s = 0}^{\infty} b_{sj} n^{- s/ \rho},\label{1.6}
\end{align}
where $\rho,\ t$ being positive integers and $\mu_{j},\ r, \ b_{sj}$ being complex numbers. In this paper, we focus on the case where $\mu_{0} = 0$ in (\ref{1.5}) and $t = 1$ in (\ref{1.6}).\\
\indent Suppose that $\{f_{n}(x)\}_{n\ge 0}$ is a P-recursive polynomial sequence with respect to the variable \(x\). According to the properties of the least common left multiple of operators, we can obtain that both $\{f_{n}^{'}(x)\}_{n\ge 0}$ and $\{f_{n}^{''}(x)\}_{n\ge 0}$ are also P-recursive polynomial sequences. Thus, they all have asymptotic expressions similar to the above. By performing a ratio operation on their asymptotic forms, we finally obtain the mean and variance of the coefficient sequence of $\{f_{n}(x)\}_{n\ge 0}$. Given the condition that the known P-recursive polynomial sequences have only real roots, we can apply this method to demonstrate that their coefficients are asymptotically normal. \\
\indent This paper is organized as follows. In Section \ref{s2}, we provide a sufficient condition for the asymptotic normality of the coefficients of a P-recursive polynomial sequence. In Section \ref{s3}, we give two examples of the asymptotic normality of specific sequences. Moreover, we list the means and variances of some other sequences.

\section{A sufficient condition for asymptotic normality}\label{s2}
In this section, we consider the asymptotic normality of the coefficient of a P-recursive polynomial sequence $\{f_{n}(x)\}_{n \ge 0}$, where $f_{n}(x)$ is expressed in the form shown in (\ref{1.1}). From (\ref{1.4}), we observe that it is asymptotically equal to
\begin{align}
e^{Q(\rho, n)(x)} s(\rho, n)(x). \label{2.1}
\end{align}
Firstly, we explore the relationships among $f_{n}(x),\ f_{n}^{'}(x)$ and $f_{n}^{''}(x)$.

\begin{lemma}\label{l2.1}
Assuming that $\{f_{n}(x)\}_{n \ge 0}$ is a P-recursive polynomial sequence in the variable $x$, then $\{f_{n}^{'}(x)\}_{n \ge 0}$ and $\{f_{n}^{''}(x)\}_{n \ge 0}$ are also.
\end{lemma}
\begin{proof}
Since $\{f_{n}(x)\}_{n\ge 0}$ is a P-recursive polynomial sequence in terms of $x$, there exists an operator ${L}_{1}$ such that 
\begin{align}
{L}_{1}f_{n}(x) = 0.\label{2.2}
\end{align}
By taking the derivative of both sides of (\ref{2.2}) with respect to $x$, we find that there exists an operator ${L}_{2}$ such that
\begin{align}
{L}_{1}f_{n}^{'}(x)+{L}_{2}f_{n}(x)=0.\label{2.3}
\end{align}
According to Theorem 8 of \cite{ore1933theory}, there exist operators $U$ and $V$ such that 
\begin{align*}
U {L}_{1} = V {L}_{2} = lclm(L_{1},L_{2}),
\end{align*}
where $lclm(L_{1},L_{2})$ denotes the least common left multiple of $L_{1}$ and $L_{2}$.
Applying the operator $V$ to both sides of (\ref{2.3}) yields
\begin{align*}
& V L_{1} f_{n}^{'}(x) + V L_{2} f_{n}(x)\\ 
= {} & V L_{1} f_{n}^{'}(x) + U L_{1} f_{n}(x) \\
= {} & V L_{1} f_{n}^{'}(x)\\
= {} & 0.
\end{align*}
Consequently, we obtain that $\{f_{n}^{'}(x)\}_{n\ge 0}$ is also a P-recursive polynomial sequence with respect to $x$. Using the same method, we can conclude that this holds for $\{f_{n}^{''}(x)\}_{n\ge 0}$.
\end{proof}

{\noindent \it Remark 1.} Suppose that $\{f_n(x)\}_{n\ge 0}$ is a P-recursive polynomial sequence. According to (\ref{2.1}), we can assume that the asymptotic expression of $f_{n}(x)$ is as follows,
\begin{align}
f_n(x) = C_{1}\cdot e^{Q_{1}(\rho, n)(x)}s_{1}(\rho, n)(x),\label{2.4}
\end{align}
where $C_1$ is a constant. \\
\indent Then we differentiate both sides of (\ref{2.4}) with respect to $x$ and obtain the asymptotic expression of $f_{n}^{'}(x)$,
\begin{align}
 f_{n}^{'}(x) =  C_{1}\cdot e^{Q_{1}(\rho, n)(x)}(Q_{1}^{'}(\rho, n)(x)s_{1}(\rho, n)(x)+s_{1}^{'}(\rho, n)(x)). \label{2.5}
\end{align}
And it can be known from Lemma \ref{l2.1} that $\{f_{n}^{'}(x)\}_{n \ge 0}$ is also a P-recursive polynomial sequence. Consequently, we can further assume that its asymptotic expression is as follows,
\begin{align}
f_{n}^{'}(x) = C_{2}\cdot e^{Q_{2}(\rho, n)(x)}s_{2}(\rho, n)(x), \label{2.6}
\end{align}
where $C_{2}$ is a constant.\\
By comparing (\ref{2.5}) and (\ref{2.6}), we can obtain that
\[Q_{1}(\rho, n)(x) = Q_{2}(\rho, n)(x)\]
and
\[C_{1}\cdot (Q_{1}^{'}(\rho, n)(x)s_{1}(\rho, n)(x)+s_{1}^{'}(\rho, n)(x)) = C_{2}\cdot s_{2}(\rho, n)(x).\]
Therefore, we can derive that the asymptotic expansion of $f_{n}^{'}(x)$,
\[f_n^{'}(x) = C_{2}\cdot e^{Q_{1}(\rho, n)(x)}s_{2}(\rho, n)(x).\]
\indent Similarly, we can also obtain the asymptotic expression of $f_{n}^{''}(x)$ as follows,
\[f_n^{''}(x) = C_{3}\cdot e^{Q_{1}(\rho, n)(x)}s_{3}(\rho, n)(x),\]
where $C_{3}$ is a constant.\\
\indent Next, we present a lemma proposed by Bender \cite{bender1973central} to determine whether the coefficients of a polynomial sequence are asymptotic normal.
\begin{lemma}\label{l2.2}
Suppose that $\{f_{n}(x)\}_{n\ge 0}$ is a sequence of real-rooted polynomial with non-negative coefficients defined by (\ref{1.1}). Let
\begin{align}
\mu_{n} = \frac{f_{n}^{'}(1)}{f_{n}(1)},\ \ \ \ \sigma_{n}^{2} = \frac{f_{n}^{''}(1)}{f_{n}(1)} + \mu_{n} - \mu_{n}^{2}.\label{2.7}
\end{align}
If $\sigma_n^2\to+\infty$ as $n\to + \infty$, then the coefficients of $f_{n}(x)$, that is, the numbers $a(n, k)$ are asymptotically normal by local and central limit theorems with mean $\mu_{n}$ and variance $\sigma_{n}^{2}$.
\end{lemma}
\indent Combined with the above Remark 1 and Lemma \ref{l2.2}, we present the main result of this paper.\\
\indent To make the theorem concise, we call that a P-recursive polynomial sequence $\{f_{n}(x)\}_{n\ge 0}$ is {\it well-defined} if the asymptotic form of $f_{n}^{(i)}(1)$ is expressed as follows,
\begin{align*}
f_{n}^{(i)}(1) = C_{i}^{'}\cdot e^{Q(\rho, n)} n^{r_{i}} \left( \sum_{s = 0}^{M} b_{s}^{(i)^{'}} n^{-s/ \rho} + o(n^{-M/ \rho})\right),
\end{align*}
where $f_{n}^{(i)}(x)$ denotes the $i$-th derivative of $f_{n}(x) (i =0, 1, 2)$, $C_{i}^{'}$ is a constant, $Q(\rho, n) = \sum_{j = 1}^{\rho} \mu_{j} n^{j/ \rho}$, $\rho, M$ are positive integers, $\mu_{j}, r_{i}, b_{s}^{(i)^{'}}$ are real numbers, $r_{1} < r_{2} < r_{3}$ and $b_{0}^{(i)^{'}} > 0$.
\begin{theorem}\label{t2.3}
Suppose that $\{f_{n}(x)\}_{n\ge 0}$ defined by (\ref{1.1}) is well-defined and is a sequence of real-rooted polynomial with non-negative coefficients. Let
\begin{align}
\mu_{n} = \frac{f_{n}^{'}(1)}{f_{n}(1)},\ \ \ \ \sigma_{n}^{2} = \frac{f_{n}^{''}(1)}{f_{n}(1)} + \mu_{n} - \mu_{n}^{2}.\label{2.8}
\end{align}
If the following two conditions are satisfied, then $a(n,k)$ is asymptotically normal.
\begin{itemize}
	\item[{\rm (1)}]
	The limits $\lim_{n \to +\infty} \frac{f_{n}^{'}(1)}{f_{n}(1)\cdot n^{r_2 - r_1}}$ and $\lim_{n \to +\infty} \frac{f_{n}^{''}(1)}{f_{n}(1)\cdot n^{r_3 - r_1}}$ both exist, denoted by $a$ and $b$ respectively.
	\item[{\rm (2)}]
There exists an integer $m$ such that $1/\rho \le m \le r_{3} - r_{1}$. For indices $0\le s \le m - 1$, the coefficients of $n^{s}$ in $\sigma_{n}^{2}$ are zero, while the coefficient of $n^{m}$ in $\sigma_{n}^{2}$ is positive.
\end{itemize}
\end{theorem}
\begin{proof}
Let $C_{i} = C_{i}^{'} b_{0}^{(i)^{'}},\ b_{s}^{(i)} = \frac{b_{s}^{(i)^{'}}}{b_{0}^{(i)^{'}}},$ then the asymptotic expression of $f_{n}^{(i)}(1)$ can be reformulated as follows,
\begin{align*}
f_{n}^{(i)}(1) = C_{i}\cdot e^{Q(\rho, n)}n^{r_{i}}\left(1+\sum_{s=1}^{M} b_{s}^{(i)} n^{-s/ \rho} + o(n^{-M/\rho})\right).
\end{align*}
Consequently, we can obtain that
\begin{align}
\mu_{n} = \frac{f_{n}^{'}(1)}{f_{n}(1)} = \frac{C_2}{C_1}\cdot n^{r_2 - r_1}\left(1 + \frac{A_{1}}{n^{1/\rho}} + \dots + \frac{A_{M}}{n^{M/\rho}} + o\left(\frac{1}{n^{M/\rho}}\right)\right), \label{2.9}
\end{align}
where $A_{s}$ is a polynomial in $b_{1}^{(0)},\ \cdots,\ b_{s}^{(0)},\ b_{1}^{(1)},\ \cdots,\ b_{s}^{(1)}$.\\
From this, we find that 
\begin{align}
\mu_{n}^{2} = \frac{C_2^2}{C_1^2}\cdot n^{2r_2 - 2r_1}\left(1 + \frac{D_{1}}{n^{1/\rho}} + \dots + \frac{D_{M}}{n^{M/\rho}} + o\left(\frac{1}{n^{M/\rho}}\right)\right)\label{2.10}
\end{align}
and
\begin{align*}
\frac{C_2}{C_1} = \lim_{n \to +\infty} \frac{f_{n}^{'}(1)}{f_{n}(1)\cdot n^{r_2 - r_1}} = a,
\end{align*}
where $D_s$ is a polynomial in $A_{1},\ \dots,\ A_{s}$.\\
\indent By similar reasoning, we establish that
\begin{align}
\frac{f_{n}^{''}(1)}{f_{n}(1)} = \frac{C_3}{C_1}\cdot n^{r_3 - r_1} \left(1 + \frac{B_{1}}{n^{1/\rho}} + \dots + \frac{B_{M}}{n^{M/\rho}} + o\left(\frac{1}{n^{M/\rho}}\right)\right)\label{2.11}
\end{align}
and
\begin{align*}
\frac{C_3}{C_1} = \lim_{n \to +\infty} \frac{f_{n}^{''}(1)}{f_{n}(1)\cdot n^{r_3 - r_1}} = b,
\end{align*}
where $B_{s}$ is a polynomial in $b_{1}^{(0)},\ \cdots,\ b_{s}^{(0)},\ b_{1}^{(2)},\ \cdots,\ b_{s}^{(2)}$.\\
Given that $r_3 > r_2 > r_1$, it follows that $r_3 - r_1 > r_2 - r_1$. Without loss of generality, we assume that $r_3 - r_1 \ge 2r_2 - 2r_1$. Consequently, there exists an integer $l_{1}$ such that $r_3 - r_1 - l_1/\rho = 2r_2 - 2r_1$, and another an integer $l_2$ such that $r_3 - r_1 - l_2/\rho = r_2 - r_1$. Thus, by (\ref{2.9}), (\ref{2.10}) and (\ref{2.11}), we derive that
\begin{align*}
\sigma_{n}^{2} &= \frac{f_{n}^{''}(1)}{f_{n}(1)} + \mu_{n} - \mu_{n}^{2}\\
&= \frac{C_3}{C_1}\cdot n^{r_3 - r_1} \left(1 + \frac{B_{1}}{n^{1/\rho}} + \dots + \frac{B_{M}}{n^{M/\rho}} + o\left(\frac{1}{n^{M/\rho}}\right)\right)\\
& + \frac{C_2}{C_1}\cdot n^{r_2 - r_1}\left(1 + \frac{A_{1}}{n^{1/\rho}} + \dots + \frac{A_{M}}{n^{M/\rho}} + o\left(\frac{1}{n^{M/\rho}}\right)\right)\\
& - \frac{C_2^2}{C_1^2}\cdot n^{2r_2 - 2r_1} \left(1 + \frac{D_{1}}{n^{1/\rho}} + \dots + \frac{D_{M}}{n^{M/\rho}} + o\left(\frac{1}{n^{M/\rho}}\right)\right)\\
&= b\cdot n^{r_3 - r_1} + b\cdot B_{1}n^{r_3 - r_1 - 1/\rho} + \cdots + b\cdot B_{l_1 - 1}n^{r_3 - r_1 - (l_1 - 1)/\rho}\\
&+ (bB_{l_1} - a^2)\cdot n^{r_3 - r_1 - l_1/\rho} + (bB_{l_1 + 1}- a^2D_{1})\cdot n^{r_3 - r_1 - (l_1 + 1)/\rho}\\
&+ \cdots + (bB_{l_2 - 1}- a^2D_{l_2 - l_1 -1})\cdot n^{r_3 - r_1 - (l_2 - 1)/\rho}\\
&+ (bB_{l_2} - a^2D_{l_2 - l_1} + a)\cdot n^{r_3 - r_1 - l_2/\rho}\\
&+ (bB_{l_2+1} - a^2D_{l_2 - l_1 +1} + a A_1)\cdot n^{r_3 - r_1 - (l_2 + 1)/\rho} + \cdots\\
&+ (bB_{(r_3 - r_1)\rho - 1} - a^2D_{(r_3 - r_1)\rho - l_1 - 1} + aA_{(r_3 - r_1)\rho - l_2 - 1})\cdot n^{1/\rho} + C + o(1).
\end{align*}
According to condition $(2)$, as $n\to + \infty$, it follows that $\sigma_{n}^{2} \to + \infty$. Therefore, $a(n, k)$ is asymptotically normal.
\end{proof}

\section{The asymptotic normality of some sequences}\label{s3}
In this section, we utilize Theorem \ref{t2.3} to demonstrate the asymptotic normality of two specific sequences. Moreover, we list the means and variances of some other combinatorial sequences. \\
\indent Recall that {\it {the hypergeometric polynomial of degree $n$}} is expressed in the following form,
\begin{align*}
_pF_q\left(
\begin{gathered}
a_1,\ a_2,\ \cdots,\ a_p\\
b_1,\ \cdots,b_q
\end{gathered}
;\,x\right) 
=
\sum_{k = 0}^{n} \frac{(a_1)_k (a_2)_k \cdots (a_p)_k}{(b_1)_k \cdots (b_q)_k}\cdot \frac{x^k}{k!},
\end{align*}
where $(a)_n = a(a + 1)\cdots (a + n - 1)$ is the Pochhammer symbol. We begin by presenting a lemma established by Driver {\it{et al.}} \cite[Theorem 7]{driver2007polya}, which proves that the roots of the hypergeometric polynomials $_3F_2$ are all real. 
\begin{lemma}\label{l3.1}
Let $a, b > 0$ and let $l, m, n \in \mathbb{N}$. Then the polynomials
\begin{align*}
_3F_2\left(
\begin{gathered}
-n,\ -m,\ a + l\\
a,\ b
\end{gathered}
;\,x\right) 
\end{align*}
has only negative real zeros.
\end{lemma}
\indent Let 
\[f_n(x) = \sum_{k = 0}^{n} f(n, k) x^k = \sum_{k = 0}^{n} \binom{n}{k}^2 \binom{n+k}{k} x^k\]
be the Ap{\'{e}}ry polynomial. Chen and Xia \cite{chen2011log} studied the 2-log-convexity of $f_{n}(1)$, while Mao and Pei \cite{mao2023asymptotic} proved the asymptotic log-convexity of Ap{\' e}ry-like numbers. Next, we will prove the asymptotic normality of $f(n,k)$.  Let the mean be $\mu_{n} = \frac{f_{n}^{'}(1)}{f_{n}(1)}$ and the variance be $\sigma_{n}^{2} = \frac{f_{n}^{''}(1)}{f_{n}(1)} + \mu_{n} - \mu_{n}^{2}$.
\begin{theorem}
The coefficients $f(n, k)$ of the Ap{\'{e}}ry polynomial $f_{n}(x)$ are asymptotically normal by local and central limit theorems with $\mu_{n} \sim \frac{-1 + \sqrt{5}}{2} n$, $\sigma_{n}^{2} \sim \frac{5 - 2\sqrt{5}}{5} n$.  
\end{theorem}
\begin{proof}
By transforming $f_{n}(x)$, we can express it as follows,
\begin{align*}
f_n(x) = {_3F_2\left(
\begin{gathered}
-n,\ -n,\ n + 1\\
1,\ 1
\end{gathered}
;\,x\right)}. 
\end{align*}
Consequently, based on Lemma \ref{l3.1}, we can know that $f_n(x)$ has only negative real roots.\\
\indent Using the Maple package {\tt{APCI}} developed by Hou \cite{hou2012maple}, we can derive the following recurrence relations, 
\begin{align*}
f_{n+2}(1) & = \frac{(11n^2+33n+25)}{(n+2)^2} f_{n+1}(1) + \frac{(n+1)^2}{(n+2)^2} f_{n}(1),\\
f_{n+2}^{'}(1) & = \frac{55n^3+187n^2+190n+48}{(5n+2)(n+2)(n+1)} f_{n+1}^{'}(1) + \frac{n(5n+7)}{(n+1)(5n+2)}f_{n}^{'}(1),\\
f_{n+2}^{''}(1) & = \frac{55n^4+154n^3+110n^2-7n-12}{n^2(n+2)(5n-1)}f_{n+1}^{''}(1) \\
& + \frac{(n+1)(n+3)(5n+4)(n-1)}{n^2(n+2)(5n-1)}f_{n}^{''}(1).
\end{align*}
Then, utilizing the Mathematica package {\tt P-rec.m} provided by Hou and Zhang \cite{hou2019asy}, we can obtain the following asymptotic expressions,
\begin{align*}
f_n(1) = {} & C_1\cdot \left(\frac{11 + 5 \sqrt{5}}{2}\right)^n n^{-1} \left(1 + \frac{-5 + \sqrt{5}}{10n} + \frac{13 - 5\sqrt{5}}{50 n^2} + o\left(\frac{1}{n^2}\right)\right),\\
f_{n}^{'}(1) = {} & C_2\cdot \left(\frac{11 + 5 \sqrt{5}}{2}\right)^n \left(1 - \frac{-1 + \sqrt{5}}{5n} + \frac{-1 + \sqrt{5}}{25 n^2} + o\left(\frac{1}{n^2}\right)\right),\\
f_{n}^{''}(1) = {} & C_3\cdot \left(\frac{11 + 5 \sqrt{5}}{2}\right)^n n \left(1 + \frac{9 - 11\sqrt{5}}{10n} - \frac{18(-3 + \sqrt{5})}{25 n^2} + o\left(\frac{1}{n^2}\right)\right),
\end{align*}
where $C_1, C_2, C_3$ are all constants. Furthermore, we find that
\begin{align}
\lim_{n\to +\infty} \frac{f_{n+1}(1)}{f_{n}(1)} = \frac{11 + 5\sqrt{5}}{2}.\label{3.1}
\end{align}
Next, based on the Maple package {\tt APCI}, we derive the following recurrence relations,
\begin{align}
f_{n}^{'}(1) & = \frac{(n+1)^2}{5n} f_{n+1}(1) - \frac{(n+1)(8n+3)}{5n} f_n(1),\label{3.2}\\
f_{n}^{''}(1) & = -\frac{(n^2+3n-1)(n+1)^2}{5n(n-1)} f_{n+1}(1) + \frac{(n+1)(13n^3+27n^2+n-3)}{5n(n-1)} f_{n}(1).\label{3.3}
\end{align}
Thus, according to Theorem \ref{t2.3}, Equations (\ref{3.1}) and (\ref{3.2}), we can conclude that
\begin{align*}
\frac{C_2}{C_1} & = \lim_{n\to  +\infty}\frac{f_{n}^{'}(1)}{f_{n}(1)\cdot n}\\
& = \lim_{n\to +\infty} \frac{(n+1)^2}{5n^2}\frac{f_{n+1}(1)}{f_{n}(1)} - \frac{(n+1)(8n+3)}{5n^2}\\
& = \frac{-1 + \sqrt{5}}{2}.
\end{align*}
Similarly, we can also derive that
\begin{align*}
\frac{C_3}{C_1} = \frac{3-\sqrt{5}}{2}.
\end{align*}
Hence, we obtain the mean as follows,
\begin{align}
\mu_{n} = {} & \frac{f_{n}^{'}(1)}{f_{n}(1)}\notag\\
= {} & \frac{-1 + \sqrt{5}}{2}\cdot n\left(1+\frac{7 - 3\sqrt{5}}{10n} + \frac{5 - 2\sqrt{5}}{25n^2} + o\left(\frac{1}{n^2}\right)\right),\label{3.4}\\
\sim {} & \frac{-1 + \sqrt{5}}{2} n \notag
\end{align}
and
\begin{align}
\frac{f_{n}^{''}(1)}{f_{n}(1)} = \frac{3-\sqrt{5}}{2}\cdot n^2\left(1+\frac{7 - 6\sqrt{5}}{5n} - \frac{2(-40 + 17\sqrt{5})}{25n^2} + o\left(\frac{1}{n^2}\right)\right).\label{3.5}
\end{align}
From (\ref{3.4}) and (\ref{3.5}), we can finally get the variance,
\begin{align*}
\sigma_{n}^{2} & = \frac{f_{n}^{''}(1)}{f_{n}(1)} + \mu_{n} - \mu_{n}^{2}\\
& = \frac{3-\sqrt{5}}{2}\cdot n^2\left(1+\frac{7 - 6\sqrt{5}}{5n} - \frac{2(-40 + 17\sqrt{5})}{25n^2} + o\left(\frac{1}{n^2}\right)\right)\\
& + \frac{-1 + \sqrt{5}}{2}\cdot n\left(1+\frac{7 - 3\sqrt{5}}{10n} + \frac{5 - 2\sqrt{5}}{25n^2} + o\left(\frac{1}{n^2}\right)\right)\\
& - \frac{3-\sqrt{5}}{2}\cdot n^2 \left(1+\frac{7-3\sqrt{5}}{5n} + \frac{67-29\sqrt{5}}{50n^2}+o\left(\frac{1}{n^2}\right)\right)
\end{align*}
\begin{align*}
& = \frac{5-2\sqrt{5}}{5}n + C +o(1)\\
& \sim \frac{5-2\sqrt{5}}{5}n.
\end{align*}
Therefore, according to Theorem \ref{2.3}, we can conclude that $f(n, k)$ is asymptotically normal.
\end{proof}
Next, we provide another example to illustrate the application of Theorem \ref{t2.3}. To begin, we present a lemma proposed by Driver {\it{et al.}} \cite[Theorem 9]{driver2007polya} that establishes the reality of the roots of the hypergeometric function $_3F_2$.
\begin{lemma}\label{l3.3}
Let $a,b > 0$ and let $l, m, n \in \mathbb{N}$. Then the polynomial
\begin{align*}
_3F_2\left(
\begin{gathered}
-n,\ -m,\ -l\\
a,\ b
\end{gathered}
;\,x\right) 
\end{align*}
has only positive real zeros.
\end{lemma}
\indent Let 
\[g_n(x) = \sum_{k = 0}^{n} g(n, k) x^k = \sum_{k = 0}^{n} \binom{n}{k}^3 x^k\]
be the Franel polynomial. And let the mean be $\mu_{n} = \frac{g_{n}^{'}(1)}{g_{n}(1)}$ and the variance be $\sigma_{n}^{2} = \frac{g_{n}^{''}(1)}{g_{n}(1)} + \mu_{n} - \mu_{n}^{2}$.
\begin{theorem}
The coefficients $g(n, k)$ of the Franel polynomial $g_{n}(x)$ are asymptotically normal by local and central limit theorems with $\mu_{n} \sim \frac{n}{2}$, $\sigma_{n}^{2} \sim \frac{n}{12}$.  
\end{theorem}
\begin{proof}
Since the Franel polynomial is defined as
\begin{align*}
g_{n}(x) = \sum_{k = 0}^{n} \binom{n}{k}^3 x^k = 
{_3F_2\left(
\begin{gathered}
-n,\ -n,\ -n\\
1,\ 1
\end{gathered}
;\,-x\right) },
\end{align*}
it follows from Lemma \ref{l3.3} that $g_{n}(x)$ has only negative real roots.\\
\indent Utilizing the Maple package {\tt APCI}, we obtain the following recurrence relations,
\begin{align*}
g_{n+2}(1) & = \frac{7n^2+21n+16}{(n+2)^2} g_{n+1}(1) + \frac{8(n+1)^2}{(n+2)^2} g_{n}(1),\\
g_{n+3}^{'}(1) & = \frac{2(3n+8)(3n^2+12n+11)}{(3n+4)(n+3)(n+2)}g_{n+2}^{'}(1)\\
& + \frac{45n^3+243n^2+432n+256}{(3n+4)(n+3)(n+2)}g_{n+1}^{'}(1)\\ 
& + \frac{8(n+1)^2(3n+7)}{(3n+4)(n+3)(n+2)}g_{n}^{'}(1),\\
g_{n+3}^{''}(1) & = \frac{3(n+2)(18n^4+100n^3+197n^2+173n+60)}{(n+3)(9n^2+11n+4)(n+1)^2}g_{n+2}^{''}(1)\\
& + \frac{3(n+2)^2(5n+7)(9n+8)}{(n+3)(n+1)(9n^2+11n+4)}g_{n+1}^{''}(1)\\
& + \frac{8(n+2)(9n^2+29n+24)}{(n+3)(9n^2+11n+4)}g_{n}^{''}(1).
\end{align*}
Next, using the Mathematica package {\tt P-rec.m}, we can derive the corresponding asymptotic expansions,
\begin{gather*}
g_{n}(1) = C_1\cdot \frac{8^n}{n}\left(1 - \frac{1}{3n} + \frac{1}{27n^2} + o\left(\frac{1}{n^2}\right)\right),\\
g_{n}^{'}(1) = C_2\cdot 8^n \left(1 - \frac{1}{3n} + \frac{1}{27n^2} + o\left(\frac{1}{n^2}\right)\right),\\
g_{n}^{''}(1) = C_3\cdot 8^n n \left(1 - \frac{2}{n} + \frac{22}{27n^2} + o\left(\frac{1}{n^2}\right)\right),
\end{gather*}
where $C_1, C_2, C_3$ are all constants. Consequently, we have that
\begin{align}
\lim_{n\to +\infty} \frac{g_{n+1}(1)}{g_{n}(1)} = \frac{g_{n+1}^{'}(1)}{g_{n}^{'}(1)} = \frac{g_{n+1}^{''}(1)}{g_{n}^{''}(1)} = 8.\label{3.6}
\end{align}
By utilizing the Maple package {\tt APCI} once more, we obtain the following recurrence relations,
\begin{align}
g_{n+1}(1) = {} & \frac{11n^2+7n+2}{(n+1)^2}g_{n}(1) - \frac{12(n-1)}{(n+1)^2}g_{n}^{'}(1) + \frac{12}{(n+1)^2}g_{n}^{''}(1),\label{3.7}\\
g_{n}(1) = {} & \frac{n+1}{6n^2+4n+1} g_{n+1}^{'}(1) + \frac{7n-5}{6n^2+4n+1} g_{n}^{'}(1) - \frac{6}{6n^2+4n+1} g_{n}^{''}(1).\label{3.8}
\end{align}
According to Theorem \ref{t2.3}, along with Equations (\ref{3.6}) and (\ref{3.7}), we get that
\begin{align*}
\frac{C_2}{C_1} = {} & \lim_{n\to +\infty} \frac{g_{n}^{'}(1)}{g_{n}(1)\cdot n}\\
= {} & \lim_{n\to +\infty} \frac{\frac{g_{n+1}(1)}{g_{n}(1)}+(1+\frac{g_{n+1}(1)}{g_{n}(1)})n}{2(1+\frac{g_{n+1}^{'}(1)}{g_{n}^{'}(1)})n}\\
= {} & \frac{1}{2}.
\end{align*}
Similarly, based on Theorem \ref{t2.3}, as well as Equations (\ref{3.6}) and (\ref{3.8}), we have
\begin{align*}
\frac{C_3}{C_1} = \frac{1}{4}.
\end{align*}
Thus, we know that the mean is given by
\begin{align}
\mu_{n} = {} & \frac{g_{n}^{'}(1)}{g_{n}(1)}
= \frac{n}{2}\label{3.9}
\end{align}
and
\begin{align}
\frac{g_{n}^{''}(1)}{g_{n}(1)} = \frac{n^2}{4}\left(1 - \frac{5}{3n} + \frac{2}{9n^2} + o\left(\frac{1}{n^2}\right)\right).\label{3.10}
\end{align}
Combining Equations (\ref{3.9}) and (\ref{3.10}), we find that the variance is
\begin{align*}
\sigma_{n}^{2} = {} & \frac{g_{n}^{''}(1)}{g_{n}(1)} + \mu_{n} - \mu_{n}^{2}\\
= {} & \frac{n^2}{4}\left(1 - \frac{5}{3n} + \frac{2}{9n^2} + o\left(\frac{1}{n^2}\right)\right) + \frac{n}{2} - \frac{n^2}{4}\\
\sim {} & \frac{n}{12}.
\end{align*}
Therefore, $g(n,k)$ is asymptotically normal.
\end{proof}
Finally, we list the asymptotic values of means and variances for some combinatorial sequences. 

\begin{table}[H]
    \renewcommand{\arraystretch}{1.2}
    \caption{The means and variances} 
    \label{table_example}
    \centering
    \begin{tabular}{|c|c|c|c|}
        \hline
        \diagbox{} & Means & Variances & References\\
        \hline
         The Narayana numbers & $\frac{n}{2}$ & $\frac{n}{8}$ & \cite{chen2020asymptotic}\\
        \hline       
        The unaerated Motzkin triangle$^{*}$ & $\frac{n}{3}$ & $\frac{n}{18}$ & \cite{chen2020analytic}\\
        \hline       
        The reversed unaerated Motzkin triangle & $\frac{n}{3}$ & $\frac{n}{18}$ & \cite{chen2020analytic}\\
        \hline
        The even Motzkin triangle$^{*}$ & $\frac{n}{3}$ & $\frac{n}{9}$ & \cite{chen2020analytic}\\
        \hline
        The reversed even Motzkin triangle$^{*}$ & $\frac{2n}{3}$ & $\frac{n}{9}$ & \cite{chen2020analytic}\\
        \hline
        The odd Motzkin triangle$^{*}$ & $\frac{n}{3}$ & $\frac{n}{9}$ & \cite{chen2020analytic}\\
        \hline
        The reversed odd Motzkin triangle$^{*}$ & $\frac{2n}{3}$ & $\frac{n}{9}$ & \cite{chen2020analytic}\\
        \hline
        The Schr\"{o}der triangle$^{*}$ & $\frac{(2-\sqrt{2})n}{2}$ & $\frac{\sqrt{2}n}{8}$ & \cite{chen2020analytic}\\
        \hline
        The reversed Schr\"{o}der triangle$^{*}$ & $\frac{\sqrt{2}n}{2}$ & $\frac{\sqrt{2}n}{8}$ & \cite{chen2020analytic}\\
        \hline      
        The generalized Narayana numbers & $\frac{n}{2}$ & $\frac{n}{8}$ & \cite{chen2022recurrences}\\
        \hline 
        The central trinomial coefficient & $\frac{n}{3}$ & $\frac{n}{18}$ & \cite{liang2023analytic}\\
        \hline 
        The Delannoy numbers & $\frac{n}{2}$ & $\frac{\sqrt{2}n}{8}$ & \cite{wang2019analytic}\\
        \hline
\end{tabular}
\end{table}
\nointerlineskip
{\noindent \it Remark 2.} The asterisk indicates that the specific numerical values were not given in the previous article. In this article, we have calculated the means and variances of the sequences. \\
{\noindent \it Remark 3.} Based on \cite{chen2020analytic}, we call an infinite lower triangular matrix $[a_{n,k}]_{n,k\ge 0}$ asymptotically normal if the numbers $a_{n,k}$ are asymptotically normal.


\begin{thebibliography}{10}

\bibitem{bender1973central}
E. A. Bender, Central and local limit theorems applied to asymptotic enumeration, J. Comb. Theory Ser. A \textbf{15} (1973), 91--111.

\bibitem{canfield1977central}
E. R. Canfield, Central and local limit theorems for the coefficients of polynomials of binomial type, J. Comb. Theory Ser. A \textbf{23} (1977), 275--290.

\bibitem{carlitz1972asymptotic}
L. Carlitz, D. C. Kurtz, R. Scoville and O. P. Stackelberg, Asymptotic properties
of Eulerian numbers, Z. Wahrscheinlichkeitstheor. Verw. Geb. \textbf{23} (1972), 47--54.

\bibitem{chen2020asymptotic}
X. Chen, J. Mao and Y. Wang, Asymptotic normality in t-stack sortable permutations, Proc.
Edinb. Math. Soc. \textbf{63} (2020), 1062--1070.

\bibitem{chen2020analytic}
X. Chen, Y. Wang and S.-N. Zheng, Analytic properties of combinatorial triangles related to Motzkin numbers, Discrete Math. \textbf{343} (2020), 112133.

\bibitem{chen2022recurrences}
X. Chen, A. L. B. Yang and  J. J. Y. Zhao, Recurrences for Callan’s generalization of Narayana polynomials, J. Syst. Sci. Complex. \textbf{35} (2022), 1573--1585.

\bibitem{chen2008limiting}
W. Y. C. Chen, C. J. Wang and L. X. W. Wang, The limiting distribution of the
coefficients of the $q$-Catalan numbers, Proc. Am. Math. Soc. \textbf{136} (2008), 3759--3767.

\bibitem{chen2010limiting}
W. Y. C. Chen and D. G. L. Wang, The limiting distribution of the $q$-derangement
numbers, Eur. J. Comb. \textbf{31} (2010), 2006--2013.

\bibitem{chen2011log}
W. Y. C. Chen and E. X. W. Xia, The 2-log-convexity of the Ap{\' e}ry numbers, Proc. Am. Math. Soc. \textbf{139} (2011), 391--400.

\bibitem{driver2007polya}
K. Driver and K. Jordaan and A. Mart{\'i}nez-Finkelshtein, P{\'o}lya frequency sequences and real zeros of some $_3F_2$ polynomials, J. Math. Anal. Appl. \textbf{332} (2007), 1045--1055.

\bibitem{feller1945fundamental}
W. Feller, The fundamental limit theorems in probability, Bull. Am. Math. Soc. \textbf{51}
(1945), 800--832.

\bibitem{flajolet2009analytic}
P. Flajolet and R. Sedgewick, Analytic combinatorics, Cambridge University Press, Cambridge, 2009.

\bibitem{harper1967stirling}
L. H. Harper, Stirling behavior is asymptotically normal, Ann. Math. Statist. \textbf{38} (1967), 410--414.

\bibitem{hou2012maple}
Q.-H. Hou, Maple package APCI, http://faculty.tju.edu.cn/HouQinghu/\\
en/index.htm, 2012.

\bibitem{hou2019asy}
Q.-H. Hou and Z.-R. Zhang, Asymptotic r-log-convexity and P-recursive sequences, J. Symb. Comput. \textbf{93} (2019), 21--33.

\bibitem{li2021analytic}
G. Li, L. L. Liu and Y. Wang, Analytic properties of sextet polynomials of hexagonal systems, J. Math. Chem. \textbf{59} (2021), 719--734.

\bibitem{liang2023analytic}
H. Liang, Y. Wang and Y. Wang, Analytic aspects of generalized central trinomial coefficients, J. Math. Anal. Appl. \textbf{527} (2023), 127424.

\bibitem{mao2023asymptotic}
J. Mao and Y. Pei, The asymptotic log-convexity of Ap{\' e}ry-like numbers, J. Difference Equ. Appl. \textbf{29} (2023), 799--813.

\bibitem{ore1933theory}
O. Ore, Theory of non-commutative polynomials, Ann. Math. \textbf{34} (1933), 480--508.

\bibitem{stanley1999enumerative}
R. P. Stanley, Enumerative Combinatorics, Volume 2, Cambridge University Press, Cambridge, 1999.

\bibitem{wang2017asymptotic}
Y. Wang, H.-X. Zhang and B.-X. Zhu, Asymptotic normality of Laplacian coefficients
of graphs, J. Math. Anal. Appl. \textbf{455} (2017), 2030--2037.

\bibitem{wang2019analytic}
Y. Wang, S.-N. Zheng and X. Chen, Analytic aspects of Delannoy numbers, Discrete Math. \textbf{342} (2019), 2270--2277.

\bibitem{zeilberger1985asymptotics}
J. Wimp and D. Zeilberger, Resurrecting the asymptotics of linear recurrences,
J. Math. Anal. Appl. \textbf{111} (1985), 162--176.

\bibitem{zhao2024asymptotic}
J. J. Y. Zhao, Asymptotic normality arising in Baxter permutations, arXiv:2410.05031, 2024.
















\end{thebibliography}
\end{document}